\documentclass[11pt, french, american]{article}
\usepackage[margin=1in]{geometry}
\usepackage{mathrsfs, bbm, dsfont}
\usepackage{
amsmath,
amsfonts,
latexsym,
amsrefs,
amssymb,
mathtools,
tikz,
tabularx,
minibox,
amsthm,
}
\usepackage{rotating}
\usepackage{bm}
\usepackage{babel}
\usepackage{floatrow}
\usepackage[T1]{fontenc}
\usepackage{graphicx}
\DeclareGraphicsExtensions{.eps}
\usepackage{url}
\definecolor{mygray}{gray}{0.2} 
\usepackage[colorlinks=true,urlcolor=gray,linkcolor=mygray,citecolor=mygray]{hyperref}
\usepackage{wrapfig}
\usepackage{titlesec}
\usepackage{titletoc}
\usepackage[title,titletoc]{appendix}
\usepackage{pdflscape}
\usepackage{afterpage}
\usepackage{caption}
\usepackage{subcaption}

\usepackage{enumerate}
\usepackage{mathrsfs}
\usepackage{wrapfig}

\usepackage{color}

\numberwithin{equation}{section}

\newcommand{\La}{\Lambda}

\newcommand{\D}{{\mathbb{D}}}

\newcommand{\al}{\alpha}

\newcommand{\ga}{{\gamma}}

\newcommand{\la}{\lambda}

\newcommand{\Ov}{\mathscr{O}}

\newcommand{\bV}{\mathbf{V}}
\newcommand{\bW}{\mathbf{W}}
\newcommand{\bB}{\mathcal{B}}

\newcommand{\bw}{\mathbf{w}}

\newcommand{\disteq}{\stackrel{d}{=}}
\newcommand{\distconv}{\xrightarrow[N \rightarrow \infty]{d}}
\newcommand{\TUE}{\mathrm{TUE}}
\newcommand{\CUE}{\mathrm{CUE}}
\newcommand{\Sph}{\mathrm{Sph}}
\newcommand{\CGE}{\mathrm{CGE}}
\newcommand{\beq}{\begin{equation}}
\newcommand{\eeq}{\end{equation}}

\newcommand{\dd}{\mathrm{d}}

\renewcommand{\epsilon}{\varepsilon}
\renewcommand{\leq}{\leqslant}
\renewcommand{\geq}{\geqslant}

\newcommand{\E}{\mathbb{E}}
\newcommand{\R}{\mathbb{R}}
\newcommand{\C}{\mathbb{C}}
\newcommand{\N}{\mathbb{N}}

\DeclareMathOperator{\tr}{Tr}
\DeclareMathOperator{\var}{Var}

\theoremstyle{plain} 
\newtheorem{theorem}{Theorem}[section]
\newtheorem*{theorem*}{Theorem}
\newtheorem{lemma}[theorem]{Lemma}
\newtheorem*{lemma*}{Lemma}

\newtheorem*{corollary*}{Corollary}
\newtheorem{proposition}[theorem]{Proposition}
\newtheorem*{proposition*}{Proposition}

\newtheorem*{definition*}{Definition}

\newtheorem*{example*}{Example}

\newtheorem*{remark*}{Remark}
\newtheorem*{remarks*}{Remarks}

\usepackage{dsfont}
\usepackage{stmaryrd}

\usepackage{tocloft}

\makeatletter
\renewcommand{\section}{\@startsection
{section}
{1}
{0mm}
{-2\baselineskip}
{1\baselineskip}
{\large\scshape\centering}} 
\makeatother

\makeatletter
\renewcommand{\subsection}{\@startsection
{subsection}
{2}
{0mm}
{-\baselineskip}
{1 \baselineskip}
{\scshape}} 
\makeatother

\makeatletter
\renewcommand{\subsubsection}{\@startsection
{subsubsection}
{3}
{0mm}
{-\baselineskip}
{- \baselineskip}
{\textit}} 
\makeatother

\title{\scshape \large On eigenvector statistics \\
in the spherical and truncated unitary ensembles}

\author{Guillaume Dubach \\
Courant Institute, NYU \\
\textit{dubach@cims.nyu.edu}}
\providecommand{\keywords}[1]{\vspace{.3in} \textbf{Keywords:} #1}

\date{}

\begin{document}

\maketitle

\begin{abstract}
We study the overlaps between right and left eigenvectors for random matrices of the spherical and truncated unitary ensembles. Conditionally on all eigenvalues, diagonal overlaps are shown to be distributed as a product of independent random variables. This enables us to prove that the scaled diagonal overlaps, conditionally on one eigenvalue, converge in distribution to a heavy-tail limit, namely, the inverse of a $\gamma_2$ distribution. These results are analogous to what is known for the complex Ginibre ensemble. We also provide formulae for the conditional expectation of diagonal and off-diagonal overlaps, with respect to all eigenvalues.

\end{abstract}

\definecolor{mygray}{gray}{0}
\tableofcontents
\definecolor{mygray}{gray}{0.2}


\newpage

\section{Introduction}

\subsection{Spherical and Truncated Unitary Ensembles}
This work considers two ensembles of random matrices defined as follows.
\begin{enumerate}
\item The spherical ensemble consists of products $G_1 G_2^{-1}$, where $G_1, G_2$ are i.i.d. complex Ginibre matrices. We denote the $N \times N$ complex Ginibre ensemble by $\CGE (N)$ and the corresponding spherical ensemble by $\Sph(N)$. The name \textit{spherical} comes from a geometric description of the eigenvalues, stated as Proposition \ref{spherical_name} and illustrated on Figure \ref{Sph_figure}.

\begin{figure}[h!]
\includegraphics[width=\textwidth]{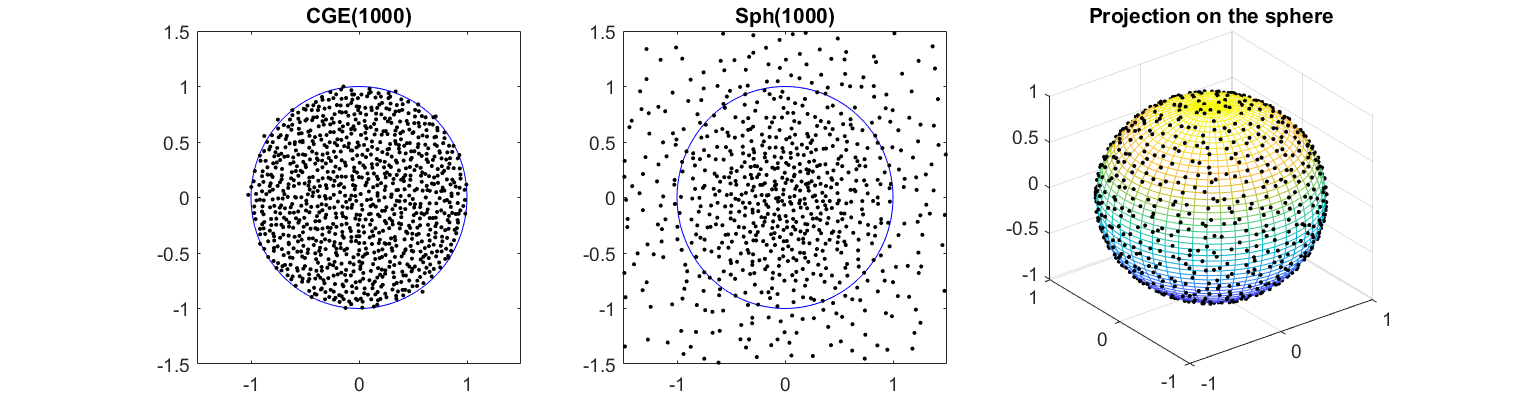}
\caption{Scaled eigenvalues of $\CGE(1000)$ and $\Sph(1000)$; the third picture is the preimage of the latter by the stereographic projection (\ref{stereo_proj}).}
\label{Sph_figure}
\end{figure}

%
%
%

\item The truncated unitary ensemble consists of truncations of unitary matrices distributed according to the Haar measure ($\CUE$). It therefore depends on two parameters determining the size of the original CUE matrix and the size of the truncation. We denote by $\TUE(N,M)$ the ensemble of truncations of size $N$ of matrices distributed according to $\CUE(N+M)$. Our results are only valid when $N \leq M$, that is, when the truncated matrix is at most half as large as the original matrix. Both parameters are assumed to go to infinity.

\begin{figure}[h!]
\includegraphics[width=\textwidth]{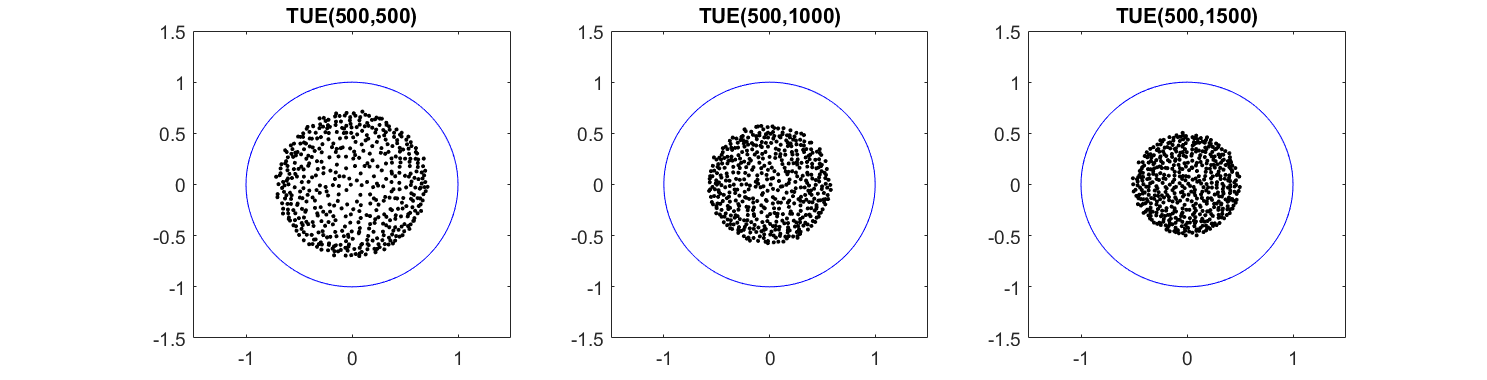}
\caption{Eigenvalues of $\TUE(N,M)$ for $N=500$ and $M=500,1000,1500$.}
\end{figure}

\end{enumerate}

%
%
%
%
%
%

The reason for treating these two ensembles in the same paper is the strong analogy between them, underlined and exemplified by \cite{ForresterKrishnapur}, that extends to the overlap distribution.  All results are presented in details for the spherical case in Section \ref{spherical_section}, while the corresponding results in the truncated unitary case are found in Section \ref{tue_section} -- with less detail whenever the two computations are exactly the same. 

\subsection{Results}
\label{overlaps_subsec}

The matrix of overlaps associated to the bi-orthogonal family of left and right eigenvectors of a non-Hermitian random matrix has been introduced and studied by Chalker \& Mehlig in \cites{ChaMeh1998,ChaMeh2000}, then more recently in a series of paper involving a variety of methods \cites{WalSta2015, BourgadeDubach, DubachQGE, Fyodorov2018, Akemannetal, CrawfordRosenthal, NowakTarnowski1, NowakTarnowski2}. It is defined as follows: for a given matrix $G \in \mathscr{M}_N(\C)$ with simple spectrum $\{ \la_1, \dots, \la_N \}$ (note that the random spectra we consider are almost surely simple), if $R_j = | R_j \rangle$ is the right eigenvector associated to $\la_j$ and $L_j = \langle L_j |$ the left eigenvector associated to the same eigenvalue, chosen such that for every $j$,
\beq\label{biorthog}
G R_j = \la_j R_j, \quad
L_j G = \la_j L_j, \quad
\langle L_i | R_j \rangle = L_i R_j = \delta_{ij},
\eeq
we define the matrix of overlaps $\Ov$ by
\beq\label{overlaps_definition}
\Ov_{i,j} = \langle L_i | L_j \rangle \langle R_j | R_i \rangle 
= (L_i L_j^*) (R_j^* R_i).
\eeq
Chalker \& Mehlig computed the conditional expectation of the overlaps in the complex Ginibre ensemble and conjectured several of their properties. For a more detailed presentation, see the introduction of \cite{BourgadeDubach} and the appendix of \cite{DubachQGE}. \medskip

The results we obtain in the spherical and truncated unitary cases are analogous to some of the results obtained in \cite{BourgadeDubach} for the complex Ginibre ensemble $\CGE(N)$. We recall these results, and point out which statement of the present paper corresponds to each one.

\begin{enumerate}

\item \textbf{A decomposition of the distribution of diagonal overlaps.} The first notable fact is that, conditionally on the spectrum $\La \in \C^N$, diagonal overlaps can be decomposed as a product of independent variables. In the complex Ginibre ensemble, Theorem 2.2 from \cite{BourgadeDubach} states that, conditionally on the event $\{ \Lambda=(\la_1, \dots, \la_N) \}$, the distribution of diagonal overlaps is given by 
\beq\label{CGE_ov11}
\Ov_{1,1}^{\CGE(N)}
\disteq
\prod_{i=2}^N \left( 1+ \frac{|Z_i|^2}{N|\la_i-\la_1|^2} \right),
\eeq
where $(Z_i)_{i=2}^N$ are i.i.d. standard complex Gaussian. Instead of Gaussian variables, the analogous statements in the spherical and truncated unitary ensembles involve i.i.d. variables whose distribution is specific to each case. Namely, in $\Sph(N)$, we have
\beq\label{sph_ov11_dist}
\Ov_{1,1}^{\Sph(N)} \disteq \prod_{k=2}^N \left( 1+  \frac{(1 +
|\la_1|^2)(1 + |\la_k|^2)}{|\la_1 - \la_k|^2} X_{N}^{(k)} \right)
\eeq
where the $ X_{N}^{(k)}$ are i.i.d. variables whose distribution is defined in (\ref{def_Xm}); and in $\TUE(N,M)$,
\beq\label{tue_ov11_dist}
\Ov_{1,1}^{\TUE(N,M)} \disteq \prod_{k=2}^N \left( 1+  \frac{(1 -
|\la_1|^2)(1 - |\la_k|^2)}{|\la_1 - \la_k|^2} Y_{M}^{(k)} \right)
\eeq
where the $ Y_{M}^{(k)}$ are i.i.d. variables whose distribution is defined in (\ref{def_Ym}). These decompositions are stated as Theorem \ref{spherical_ov11} and \ref{tue_ov11} respectively.

\item \textbf{A limit theorem for diagonal overlaps.}
In the complex Ginibre ensemble, Theorem 1.1 from \cite{BourgadeDubach} states that conditionally on the event $\{ \la_1 = z \}$ with $z \in \D$, the scaled diagonal overlap $\Ov_{1,1}$ converges to the inverse of a $\ga_2$ distribution: 
\beq
\frac{1}{N} \Ov_{1,1}^{\CGE(N)} \distconv
(1 - |z|^2) \frac{1}{\ga_2}.
\eeq

This heavy-tail limit\footnote{or heavy-tail distributions of the same family, such as $\ga_1^{-1}$ and $\ga_4^{-1}$ (see \cite{Fyodorov2018} and \cite{DubachQGE}).} appears to be universal. 

In particular, the exact same convergence holds at the origin for the spherical and truncated unitary ensembles, which is stated as Proposition \ref{spherical_gamma2lim} and \ref{tue_gamma2lim} respectively. Unlike the complex Ginibre case, where $\Ov_{1,1}^{-1}$ follows a beta distribution when $\{ \la_1 = 0 \}$, the distribution of the overlap for fixed $N$ does not take an especially simple form here; nevertheless, the asymptotical result can be worked out in an analogous way.

\begin{figure}[h!]
\includegraphics[width=\textwidth]{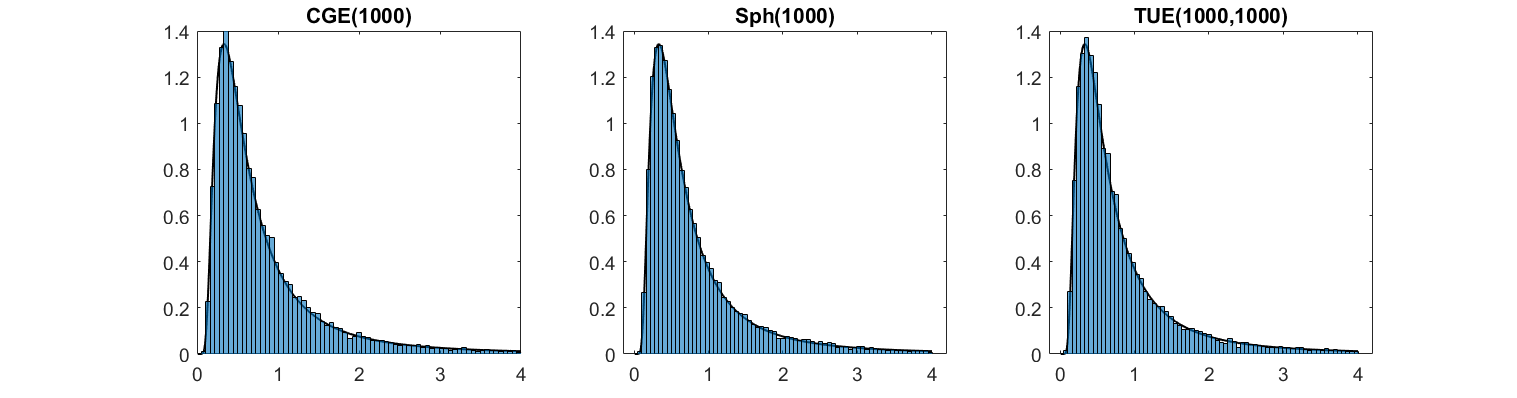}
\caption{Histograms of scaled diagonal overlaps for $\CGE(N), \Sph(N)$ and $\TUE(N,N)$ respectively, with $N=1000$ and over $30$ experiments (for each experiment, the overlaps of all eigenvalues in a given domain, chosen arbitrarily inside the bulk, have been considered).}
\label{Histograms_ov11}
\end{figure}

The specific structure of the spherical ensemble allows one to extend this result to the whole complex plane, yielding the following Theorem. 

\begin{theorem}\label{spherical_limit}
Conditionally on the event $\{ \la_1 = z \}$ with $z \in \C$,
\beq
\E \left( \Ov_{1,1}^{\Sph(N)} \right) = N \quad \text{and} \quad
\frac{1}{N} \Ov_{1,1}^{\Sph(N)} \distconv
 \frac{1}{\ga_2}.
\eeq
\end{theorem}

It is to be expected that a similar statement holds for $\TUE(N,M)$ in the bulk of its limit density of eigenvalues, with a scaling parameter coherent with the expressions of \cite{NowakTarnowski1}. \medskip

\item \textbf{Conditional expectations of overlaps.} In the complex Ginibre ensemble, it follows from (\ref{CGE_ov11}) that the expectation of diagonal overlaps of $\CGE(N)$ takes the following form:

\beq
\E_{\Lambda}^{\CGE(N)} \left( \Ov_{1,1} \right)
= \prod_{k =2}^N \left( 1+ \frac{1}{N |\la_i - \la_k |^2}\right),
\eeq
which had been obtained earlier by Chalker \& Mehlig \cites{ChaMeh1998,ChaMeh2000} by a direct computation.

Analogous identities derive from equations (\ref{sph_ov11_dist}) and (\ref{tue_ov11_dist}); they are stated in Theorem \ref{spherical_ov11} and \ref{tue_ov11} respectively. Moreover, expressions of the same kind can be obtained for off-diagonal overlaps, although no decomposition in independent variables is known in that case. In the Ginibre ensemble, this yields

\beq
\E_{\Lambda}^{\CGE(N)} \left( \Ov_{1,2} \right)
= - \frac{1}{|\la_1 - \la_2|^2}
\prod_{k =3}^N \left( 1+ \frac{1}{N (\la_1 - \la_k)(\overline{\la_2 -
\la_k})}\right).
\eeq
The analogous results for $\Sph(N)$ and $\TUE(N,M)$ are stated as Theorem \ref{spherical_ov12} and \ref{tue_ov12} respectively. \medskip

\item \textbf{Conditional expectation of mixed moments.} The conditional expectation of $\tr G^* G$ with respect to $\La$ also exhibits a remarkable decomposition in all three ensembles. One reason for considering this particular quantity, which is the simplest 'mixed moment`, is that it is obtained from the eigenvalues and the overlaps by the identity:
\beq\label{linear_rel}
\tr G G^* = \sum_{i,j=1}^N \la_i \overline{\la_j} \Ov_{i,j}.
\eeq
More general mixed moments are linked to the generalized overlaps considered in \cite{CrawfordRosenthal} by similar relations.

In the complex Ginibre case, it suffices to write
\beq\label{Gin_mixedmom}
\tr GG^* = \tr(TT^*) = \sum_{i\leq j} |T_{i,j}|^2
=
\sum_{i=1}^N |\la_i |^2 + \sum_{i < j} |T_{i,j}|^2.
\eeq
The conditional expectation follows immediately, using the fact that the upper-diagonal entries of the Schur transform are Gaussian and independent of the eigenvalues.

The spherical and truncated unitary ensembles yield slightly more intricate expressions, stated as Proposition \ref{spherical_mixedmom} and \ref{tue_mixedmom} respectively. \medskip
\end{enumerate}

We summarize all results relative to (iii) and (iv) in the table below, Section \ref{synoptic_sec}. It follows from (\ref{linear_rel}) that the third column is related to the first two by elementary linear relations -- a fact which is not directly seen from the quenched expressions.

\afterpage{%
\clearpage
\thispagestyle{empty}
\begin{landscape}
\subsection{Synoptic table of conditional expectations}\label{synoptic_sec} \small
\begin{center}
\begin{tabular}{|l|c|c|c|}
\cline{2-4}
\multicolumn{1}{c|}{} &  &  & \\
\multicolumn{1}{c|}{}
& $\E_{\Lambda} \left( \Ov_{1,1} \right)$
& $\E_{\Lambda} \left( \Ov_{1,2} \right)$
& $\E_{\La} \left( \frac{1}{N} \tr GG^* \right)  $  \\
\multicolumn{1}{c|}{} &  & &  \\
  \hline
  &  &  &  \\
Complex Ginibre &  &  &  \\
$\CGE(N)$ & $\displaystyle{\prod_{k=2}^N \left(1+ \frac{1 }{N |\la_1 -
\la_k|^2} \right)}$
& $\displaystyle{- \frac{1}{N|\la_1 - \la_2|^2}  \prod_{k=3}^N \left( 1+
\frac{1}{N (\la_1 - \la_k)(\overline{\la_2 - \la_k})} \right)}$
& $\displaystyle{\frac{1}{N} \sum_{k=1}^N |\la_k|^2 + \frac{N-1}{2N}}$ \\
 &  &  &  \\
  & Chalker \& Mehlig \cites{ChaMeh1998, ChaMeh2000}
  & Chalker \& Mehlig \cites{ChaMeh1998, ChaMeh2000}
  & From eq. (\ref{Gin_mixedmom})\\
 & & & \\
  \hline
 &  &  &  \\
{\footnotesize Spherical Ensemble} & & & \\
 $\Sph(N)$  & $\displaystyle{\prod_{k=2}^N \left(1+ \frac{(1 + |\la_1|^2)
(1 + |\la_k|^2)}{N |\la_1 - \la_k|^2} \right)}$
& $\displaystyle{- \frac{1}{N|\la_1 - \la_2|^2}  \prod_{k=3}^N \left( 1+
\frac{(1 + \la_1 \overline{\la_2}) (1 + |\la_k|^2)}{N (\la_1 -
\la_k)(\overline{\la_2 - \la_k})} \right)}$
& $\displaystyle{ \prod_{k=1}^N \left(1+ \frac{1 + |\la_k|^2}{N} \right)
-2 }$  \\
 &  &  &  \\
 &  Theorem \ref{spherical_ov11}
  & Theorem \ref{spherical_ov12}
  & Proposition \ref{spherical_mixedmom}\\
 & & & \\
 \hline
 &  &  &  \\
{\footnotesize Truncated Unitary} &  &  &  \\
$\TUE(N,M)$ & $\displaystyle{\prod_{k=2}^N \left(1+ \frac{(1 -
|\la_1|^2)(1 - |\la_k|^2)}{M |\la_1 - \la_k|^2} \right)}$
& $\displaystyle{- \frac{1}{M |\la_1 - \la_2|^2}  \prod_{k=3}^N \left( 1+
\frac{(1 - \la_1 \overline{\la_2}) (1 - |\la_k|^2)}{M (\la_1 -
\la_k)(\overline{\la_2 - \la_k})} \right)}$
& $\displaystyle{\prod_{i=1}^N \left( 1 + \frac{1- |\la_i|^2}{M} \right)
- \left( 1+ \frac{N}{M}\right)}$ \\
 &  &  &  \\
  &  Theorem \ref{tue_ov11}
    & Theorem \ref{tue_ov12}
    & Proposition \ref{tue_mixedmom}\\
 & & & \\
\hline
\end{tabular}
\captionof{table}{Quenched expectations.}
\end{center}
The above table contains the expression of quenched expectations (that is, conditional expectations with respect to $\La=\{ \la_1, \dots, \la_N \}$) of overlaps in the three relevant integrable ensembles. The first column is the consequence of an identity in distribution involving independent variables; not so for off-diagonal overlaps (second column). The third column is related to the first two by the linear relations implied by (\ref{linear_rel}).
\end{landscape}
\clearpage
}

\newpage

\subsection{Method, notations and conventions}

\subsubsection{Overlaps and Schur form.}

We recall here only what is needed in order to follow the method we apply to the spherical and truncated unitary cases. \medskip

We first note that the conditions (\ref{biorthog}) can be achieved by choosing $R_i$ as the columns of $P$ and $L_i$ as the rows of $P^{-1}$ for a given diagonalization $G = P \Delta P^{-1}$; the overlaps are independent of this choice. Moreover, overlaps are unchanged by an unitary change of basis, and therefore one can study directly the overlaps of the Schur form
\beq\label{schur_form}
T = U^* G U =  \left( \begin{array}{cccc}
\lambda_1 & T_{1,2} & \dots & T_{1,N} \\
0 & \lambda_2 & \dots & T_{2,N} \\
 \vdots & \ddots & \ddots & \vdots \\
0 & \dots & 0 & \lambda_N \\
 \end{array} \right).
\eeq
By exchangeability of the eigenvalues, we can also limit ourselves to studying the variables $\Ov_{1,1}$ and $\Ov_{1,2}$, whose definitions only involve the first two left and right eigenvectors of $T$, chosen such that
\begin{align*}
R_1=(1,0, \dots, 0)^{\rm t}, & \qquad R_2=(a,1,0,\dots, 0)^{\rm t}, \\
L_1=(b_1, \dots, b_N), & \qquad
L_2=(d_1,\dots, d_N).
\end{align*}
Biorthogonality (\ref{biorthog}) gives $b_1=1$, $d_1=0$, $d_2=1$ and
$a=-b_2$.
Thanks to the upper-triangular form of $T$, the coefficients $b_i, d_i$ are obtained according to a straightforward recurrence. Indeed, if we consider the sequences of sub-vectors:
\begin{align*}
B_k &= (1, b_2,\dots ,b_k) \qquad \text{so that } L_1=B_N, \\
D_k &= (0, 1, d_3,\dots , d_k) \qquad \text{so that } L_2=D_N, \\
u_k &= (T_{1,k},\dots ,T_{k-1,k})^{\rm t} \qquad \text{(subset of the
$k$th column of $T$)}.
\end{align*}
The recurrence formula is
\beq\label{vect_recurrence}
\left\{
\begin{array}{rl}
b_{n+1} & = \frac{1}{\la_1 - \la_{n+1}} B_n u_{n+1},\quad n\geq 1, \\
& \\
d_{n+1} & = \frac{1}{\la_2 - \la_{n+1}}
D_n u_{n+1},\quad n\geq 2.
\end{array}
\right.
\eeq
The first overlaps, according to (\ref{overlaps_definition}), are then given by the expressions
\begin{equation}
\Ov_{1,1} = \sum_{i=1}^N |b_i|^2, \qquad
\Ov_{1,2} = - \overline{b_2} \sum_{i=1}^N b_i \overline{d_i}.
\end{equation}
The reason why the recurrence (\ref{vect_recurrence}) leads to a decomposition in distribution (resp. a decomposition of the conditional expectation with respect to all eigenvalues) of the overlaps in different ensembles is that the distribution of the Schur form is known and allows to perform such a computation explicitly. For instance, in the complex Ginibre case, the upper-triangular entries $(T_{i,j})_{i<j}$ are i.i.d. complex Gaussian variables with variance $1/N$, so that $u_{k+1}$ is a $k$-dimensional Gaussian vector with independent coordinates, and independent of $u_2, \dots, u_k$. The Schur forms of $\Sph(N)$ and $\TUE(N,M)$ have explicit densities expressed in the form of a determinant; a structure which allows an analogous analysis.

\subsubsection{Notations and conventions.}
Throughout the paper, $N$ is the size of the system (i.e. the number of eigenvalues);  the spectrum is $\La = \La_N = (\la_1, \dots, \la_N)$. For any $n \leq N$, we denote by $T_n$ the $n \times n$ top-left submatrix of the Schur form $T$, and by $u_n$ the first $n-1$ coordinates of the last column vector of $T_n$, so that
$$
T_n = 
\left(
\begin{array}{cc}
    T_{n-1} & u_n  \\
    0 & \la_n 
\end{array}
\right), \qquad T=T_N.
$$
 $\E_{A}$ denotes the conditional expectation with respect to $A$ (if $A$ is a random variable or a sigma algebra), or the expectation for the conditional probability (if $A$ is an event); the context should prevent any ambiguity to arise. In particular, $\E_{\Lambda}$ is the conditional expectation with respect to the spectrum $\La$. When conditioning on $\La$, we will also use the following filtration, adapted to the nested structure of the Schur transform: 
$$
\mathscr{F}_n=\sigma\left(u_k, 2 \leq k\leq n \right)=\sigma\left(T_{i,j},1\leq i < j\leq n\right).
$$
(This convention differs from the one chosen in \cite{BourgadeDubach}. In particular, $b_2=\frac{T_{12}}{\lambda_1-\lambda_2}\in\mathscr{F}_2$, and $\mathscr{F}_1$ is trivial.)
With any suitable function $V$, the generalized Gamma and Meijer functions are defined as
$$
\Gamma_V(\alpha) := \int_{\R_+} t^{\al - 1} e^{-V(t)} \dd t,
\qquad
G_{V} (k) = \Gamma_V(1) \cdots \Gamma_V(k).
$$
We also define the partial sums
$$
e_V^{(m)}(X) = \sum_{k=0}^m \frac{X^k}{\Gamma_V (k+1)},
$$
and the generalized Gamma distributions $\ga_V(\al)$, with density
\beq\label{gammaVdist}
\frac{1}{\Gamma_V(\al)}
t^{\al-1} e^{-V(t)} \mathds{1}_{\R_+}
\eeq
with respect to the Lebesgue measure. We will use the fact, established for instance in \cite{DubachPowers, HKPV, Kostlan}, that a point process in $\C$ with joint density given by
\beq\label{density_general_V}
\frac{1}{Z_N} \prod_{1 \leq i<j \leq N} |\la_i - \la_j|^2 e^{- \sum_{i=1}^N V(|\la_i|^2)} 
\eeq
where $ Z_N = G_V(N)$, is such that the following identity in distribution holds:
\begin{proposition}[Kostlan's property]\label{Kostlan}
$\{|\la_1|^2, \dots, |\la_N|^2 \} \disteq \{ \ga_{V}(1), \dots, \ga_{V}(N) \} $ where the latter variables are independent, and $\ga_V(k)$ is distributed according to (\ref{gammaVdist}) with $\al=k$.
\end{proposition}

What we need here is a specific form of Kostlan's property, obtained by applying Proposition \ref{Kostlan} to the conditioned measure.

\begin{proposition}\label{Kostlan_conditioned} Conditionally on the event $\{ \la_1 = 0 \}$, $\{|\la_2|^2, \dots, |\la_N|^2 \} \disteq \{ \ga_{V}(2), \dots, \ga_{V}(N) \} $ where the latter variables are independent, and $\ga_V(k)$ is distributed according to (\ref{gammaVdist}) with $\al=k$.
\end{proposition}

In the complex Ginibre ensemble, the scaled variables $N \gamma_V$ follow usual gamma distributions. \medskip

Other notations or conventions relatives specifically to the spherical or truncated unitary case are mentioned in the corresponding section.

\newpage
\section{Spherical ensemble}\label{spherical_section}

This section contains the proof of all claims related to the spherical ensemble $\Sph(N)$. These proofs rely on a few estimates that are found in Subsection \ref{spherical_estimates}.

\subsection{Schur form and eigenvalues}

We first present a few general results in order to illustrate the method; the tools and definitions that follow are specific to the spherical case. We recall that the Schur transfom $T$ of a matrix from $\Sph(N)$ is distributed with density proportional to
\beq\label{sph_Schur_density}
\prod_{1 \leq i<j \leq N} |\la_i - \la_j|^2
\frac{1}{\det(I_N + T T^*)^{2N}}
\eeq
with respect to the Lebesgue measure on all complex matrix elements, diagonal ($\dd \La = \dd \la_1 \cdots \dd \la_N$) and upper-triangular ($\dd u_2 \cdots \dd u_n$).  \medskip
We introduce the Hermitian, definite-positive matrices
\beq\label{def_HnSn}
H_n := I_n+T_n T_n^*,
\qquad
S_{n-1}:= (1+|\la_n|^2)^{1/2} H_{n-1}^{1/2}.
\eeq
The following lemma is the essential tool used in \cite{ForresterKrishnapur}.

\begin{lemma}\label{FK_trick} The determinant of $H_n = I_n+T_n T_n^*$ can be reccursively decomposed as
\begin{align}
\det(H_n)
= &
(1+|\la_n|^2)
\det(H_{n-1})
\left( 1+ \frac{1}{1+|\la_n|^2} u_n^* H_{n-1}^{-1} u_n \right).
\end{align}
\end{lemma}

\begin{proof}
We first write
\begin{equation*}
\det ( H_n )
=
\left|
\begin{array}{cc}
    I_{n-1}+T_{n-1}T_{n-1}^* + u_n u_n^* & \overline{\la_n} u_n  \\
    \la_n u_n^* & 1+|\la_n|^2 
\end{array}
\right|.
\end{equation*}
Elementary operations on columns brings this matrix to an upper-triangular form, so that
\begin{align*}
\det (H_n)
& =
(1+|\la_n|^2 )
\det \left( I_{n-1}+T_{n-1}T_{n-1}^* + \frac{1}{1+|\la_n|^2} u_n u_n^* \right) \\
& = (1+|\la_n|^2 ) \det (H_{n-1})
\det \left( I_{n-1} + \frac{1}{1+|\la_n|^2} u_n u_n^* H_{n-1}^{-1} \right).
\end{align*}
The claim follows by Sylvester's identity, $\det(I+AB) = \det (I+BA)$.
\end{proof}

For any $p > n$, we denote by $\bV_p^{(n)}$ a random vector with density
\beq\label{def_Vpn}
\frac{1}{C_{n,p}}
\frac{1}{(1+v^* v)^p}
\eeq
with respect to the Lebesgue measure on $\C^n$; the value of $C_{n,p}$ is given by (\ref{sph_ct0}). For any $m \geq 0$, we denote by $X_m$ a real random variable with density
\beq\label{def_Xm}
\frac{m+1}{(1+x)^{m+2}} \mathds{1}_{\R_+}
\eeq
with respect to the Lebesgue measure. In particular $\E X_m = \frac{1}{m}$, and if $v_i$ is a coordinate of $\bV_p^{(n)}$, it follows from Lemma \ref{spherical_distributions} that
$$
|v_i|^2 \disteq X_{p-n-1}.
$$
Note that the i.i.d. variables that appear in Theorem \ref{spherical_ov11} follow the distribution of $X_m$ with $m=N$.

\begin{lemma}\label{sph_key_prop}
Identity holds between the following expressions, for $p \geq n$ and $f, g$ integrable functions of the matrix elements: 
\begin{align*}
\int  \frac{f(\La_n,u_2, \dots, u_{n-1}) g (u_{n})}{\det ( H_n )^{p}}
 \dd T_n
 =
C_{n-1,p}
\int \frac{f(\La_n,u_2, \dots, u_{n-1}) \E \left( g(S_{n-1} \bV_p^{(n-1)}) \right) }{(1+|\la_n|^2)^{p-n+1} \det ( H_{n-1} )^{p-1}}
\dd T_{n-1} \dd \lambda_{n},
\end{align*}
where
$ H_{n}, S_{n-1}, \bV_p^{(n)} $
are defined in (\ref{def_HnSn}) and (\ref{def_Vpn}).
\end{lemma}

\begin{proof} Lemma \ref{FK_trick} and the change of variable $u_n = S_{n-1} v_n$ bring the left hand side to the form
$$
\int  \frac{f(\La_n,u_2, \dots, u_{n-1}) g (S_{n-1} v_n)}{(1+|\la_n|^2)^{p-n+1} \det ( H_{n-1} )^{p-1} (1+ v_n^* v_n)^p}
\dd T_{n-1} \dd v_n \dd \la_n.
$$
Recall that $u_n$, and therefore $v_n$, are column vectors of size $n-1$. The claim follows by definition of the random vector $\bV_p^{(n-1)}$.
\end{proof}

A first relevant fact that can be deduced from the above Lemma is the distribution of every top-left submatrix of the Schur form $T$.

\begin{proposition}\label{sph_subschur}
Conditionally on $\La$ and for $2 \leq n \leq N$, the submatrix $T_n$ of the Schur transform is distributed with density proportional to
\beq
\frac{1}{\det(I_n + T_n T_n^*)^{N+n}}.
\eeq
with respect to the Lebesgue measure on upper-triangular matrix elements ($\dd u_2 \cdots \dd u_n$).
\end{proposition}

\begin{proof}
The claim is known for $n=N$. We deduce it for all $n$ by a backward recurrence; indeed, as long as $n-1 \geq 2$, the claim for $n-1$ follows from the claim for $n$ by Lemma \ref{sph_key_prop} with $g=1$ and generic $f$.
\end{proof}

We can also derive the joint eigenvalue density of the spherical ensemble from the density of its Schur form, as was done in \cite{ForresterKrishnapur}.

\begin{theorem}\label{sph_joint_density}
The joint density of eigenvalues for the spherical ensemble is proportional to
\beq\label{sph_joint_dens}
\prod_{i<j} |\la_i - \la_j|^2
\frac{1}{\prod_{i=1}^N (1 + |\la_i|^2)^{N+1}} 
\eeq
with respect to the Lebesgue measure on $\C^N$.
\end{theorem}

\begin{proof} Let $h$ be a bounded and continuous function of the spectrum $\La_n$. We use Lemma \ref{FK_trick} with $p=2N$, $g=1$ and
$$
f_0(\La_n ,u_2, \dots, u_n) := \prod_{i<j} |\la_i - \la_j|^2 h(\La_n),
$$
which yields
$$
\E \left( h(\La_n) \right)
=
C_{N-1,2N}
\int \frac{f_n(\La_n,u_2, \dots, u_{n-1})}{(1+|\la_n|^2)^{N+1} \det ( H_{n-1} )^{2N-1}}
\dd T_{n-1} \dd \lambda_{n}
$$
we then use Lemma \ref{FK_trick} again with
$$
f_{n-1}(\La_{n-1},u_2, \dots, u_{n-1}) := \int \frac{f_n(\La_n,u_2, \dots, u_{n-1})}{(1+|\la_n|^2)^{p-n+1}} \dd \lambda_{n},
$$
and so on; this recurrence leads to the expression
$$
\E \left( h(\La_N) \right)
=
C
\int \prod_{i<j} |\la_i - \la_j|^2  \frac{ h(\La_N)}{\prod_{i=1}^N (1 + |\la_i|^2)^{N+1}} \dd \La_N,
$$
which is equivalent to the claim.
\end{proof}

Theorem \ref{sph_joint_density} can be rephrased by saying that the eigenvalues of $\Sph(N)$ are distributed according to (\ref{density_general_V}) with potential $V(t)= (N+1) \ln (1+t)$. A straightforward computation shows that
\beq\label{spherical_gamma}
\ga_V(\al) \disteq \frac{1}{\beta_{N+1-\al, \al}} - 1.
\eeq

\subsubsection*{Origin of the name spherical.} The stereographic projection from $\mathbb{S}^2$ to $\C$ is defined by
\beq\label{stereo_proj}
\rho(\bw) = \tan\left(\frac{\phi}{2} \right) e^{i \theta},
\qquad \text{where} \quad
{ \bw } =
\left(
\begin{array}{c}
\sin \phi \cos \theta \\
\sin \phi \sin \theta \\
\cos \phi
\end{array}
\right)
\in \mathbb{S}^2,
\eeq
and its inverse map from $\C$ to $\mathbb{S}^2$ is given by :
\beq\label{stereo_proj_inv}
p(\lambda)=\frac{1}{1+|\lambda|^2}
\left(
\begin{array}{c}
2{\rm Re} \lambda \\
2{\rm Im} \lambda \\
|\lambda|^2 -1
\end{array}
\right)
\in \mathbb{S}^2,
 \qquad \text{where } \la \in \C.
\eeq

The reason for the name \textit{spherical} is that the following identity in distribution holds.
\begin{proposition}\label{spherical_name} Let $(\bw_1, \dots, \bw_N) = (p(\la_1), \dots, p(\la_N))$ be the images of the eigenvalues of $G_1 G_2^{-1}$ by the map (\ref{stereo_proj_inv}).
This point process on $\mathbb{S}^2$ has joint density proportional to
\beq\label{point_process_1}
\prod_{i<j} \| \bw_i - \bw_j \|_{\R^3}^2.
\eeq
\end{proposition}

In other terms, the eigenvalues of $\Sph(N)$ can be described as the stereographic projection of a one-component plasma on $\mathbb{S}^2$, with respect to a uniform potential\footnote{Note that the appropriate convention for the stereographic projection here is such that the unit circle is mapped to the equator of $\mathbb{S}^2$. In particular, the average proportion of eigenvalues of $\Sph(N)$ falling in the unit disk is $\frac12$.}.

\begin{proof}
This is obtained by a change of variable applied to the density (\ref{sph_joint_dens}), noting that
\beq\label{stereo_identity}
\|p(\la) - p (\mu)\|_{\R^3}^2
=
\frac{4 |\la - \mu|^2}{ (1+|\la |^2)(1+|\mu |^2)}
\eeq
and that the Jacobian of $p$ at $\la$ is $\frac{4}{(1+|\la|^2)^2}$.
\end{proof}

\subsection{Distribution and conditional expectation of overlaps}

We now give the proof of the claims concerning diagonal and off-diagonal overlaps in the spherical ensemble. Some results hold conditionally on the whole spectrum $\La_N$, whereas others only imply a condition on one eigenvalue.

\begin{theorem}\label{spherical_ov11} Conditionally on $\{ \La=(\la_1, \dots, \la_N) \} $, the diagonal overlaps of $\Sph(N)$ are distributed as
\beq
\Ov_{1,1} \disteq \prod_{k=2}^N \left( 1+  \frac{(1 +
|\la_1|^2)(1 + |\la_k|^2)}{|\la_1 - \la_k|^2} X_{N}^{(k)} \right)
\eeq
where the $ X_{N}^{(k)}$ are i.i.d. distributed according to (\ref{def_Xm}) with $m=N$. In particular, the quenched expectation is given by
\beq
\E_{\La} \left( \Ov_{1,1} \right) = \prod_{k=2}^N \left( 1+  \frac{(1 +
|\la_1|^2)(1 + |\la_k|^2)}{N |\la_1 - \la_k|^2}\right).
\eeq
\end{theorem}

\begin{proof} For $1 \leq d \leq N$, we define the partial sums
$$
\Ov_{1,1}^{(d)} := 
\sum_{i=1}^d |b_i|^2.
$$
It follows from the general facts presented in Section \ref{overlaps_subsec} that $\Ov_{1,1}^{(1)} = 1$, and for any $d \geq 1$,
$$
\Ov_{1,1}^{(d+1)} = \Ov_{1,1}^{(d)} + |b_{d+1}|^2 = \Ov_{1,1}^{(d)} \left(1 + \frac{1}{|\la_1 - \la_{d+1}|^2} \frac{ | B_d u_{d+1}  |^2 }{\| B_d \|^2} \right) 
$$
In order to characterize the distribution of this factor, we use our preliminary results in the following order:  
\begin{itemize}
\item Proposition \ref{sph_subschur} gives the distribution of $T_{d+1}$, so that $p=N+d+1$ in the following steps.
\item Lemma \ref{sph_key_prop} with $n=d+1$, generic $f$ and $g(u_{d+1}) := h(| B_d u_{d+1}  |^2)$ with generic $h$ gives that
$$
| B_d u_{d+1}  |^2 \disteq | B_d S_d \bV_{N+d+1}^{(d)} |^2
$$
and is independent of $\mathcal{F}_d$.
\item Lemma \ref{spherical_distributions} with $a=b= B_d^*$ and $S=S_{d}$ yields
\beq
| B_d u_{d+1} |^2 \disteq 
 \| S_d 
B_d^* \|^2 X_{N}
= (1+|\la_{d+1}|^2) ( B_d (I_d+ T_d T_{d}^*) 
B_d^* ) X_{N}
\eeq
where $X_{N}$ is distributed according to (\ref{def_Xm}) with parameter $m=N$, and independent of $\mathscr{F}_d$. 
\end{itemize}
We notice that, as $T$ is triangular and $T_d, B_d$ are obtained from $T$ and $L_1$,
\beq
B_d T_{d}
= \la_1 B_d
\eeq
which implies that $B_d (I_{d} + T_d T_{d}^*) B_d^*  = (1 + |\la_1|^2) \| B_{d} \|^2$. It follows that
$$
\Ov_{1,1}^{(d+1)} 
\disteq
\Ov_{1,1}^{(d)} \left(1 + \frac{(1 + |\la_1|^2)(1 + |\la_{d+1}|^2)}{|\la_1 - \la_{d+1}|^2} X_{N} \right),
$$
where $X_{N} $ is independent of $ \mathscr{F}_d$; we denote this variable by $X_N^{(d+1)}$ in order to avoid confusion between the different variables $X_N$. This implies the claim, as $\Ov_{1,1} = \Ov_{1,1}^{(N)}$.
\end{proof}

Diagonal overlap are (deterministically) larger than one, and typically of order $N$. The following proposition states that in the spherical ensemble the expectation of the diagonal overlap for an eigenvalue conditioned to be at the origin is exactly $N$, as is also the case in the complex Ginibre and truncated unitary ensembles.

\begin{proposition}\label{sph_expectation_N}
Conditionally on $\{ \la_1=0 \}$, the expectation of the diagonal overlap $\Ov_{1,1}$ in the spherical ensemble $\Sph(N)$ is
$$
\E_{\{ \la_1=0 \}} \Ov_{1,1} = N.
$$
\end{proposition}

\begin{proof}
We know from Proposition \ref{Kostlan} that the squared radii are distributed like independent variables with distributions $\gamma_{V,k}$ with $V (x)= (N+1) \log(1+x)$ and $2 \leq k \leq N$.
We have
$$
\E_{\{ \la_1 = 0 \}} \Ov_{1,1} 
= \prod_{k=2}^N \E \left( 1 + \frac{1}{N} + \frac{1}{N \gamma_{V,k}} \right),
$$
and according to Lemma \ref{spherical_gammagen},
$$
\E \left( \frac{1}{ \gamma_{V,k}} \right)
= \frac{\beta (N+2 -k, k-1) }{ \beta(N+1-k, k) }
= \frac{N+1-k}{k-1},
$$
so that the expectation is given by the telescopic product
$$
\E_{\{ \la_1 = 0 \}} \Ov_{1,1}
= \prod_{k=2}^N \frac{k}{k-1}
=N
$$
as was claimed.
\end{proof}

\begin{proposition}\label{spherical_gamma2lim}
Conditionally on $\{ \la_1=0 \}$, the following convergence in distribution takes place:
$$
\frac{1}{N} \Ov_{1,1} \distconv \frac{1}{\gamma_2}.
$$
\end{proposition}

The proof relies on the following elementary Lemma.

\begin{lemma}\label{sequence_exists}
Let $\left( u_{k,n}^{(m)} \right)_{1 \leq k \leq n}^{m \in \N}$ be a countable family of double-indexed real positive sequences such that
$$
\forall m,k \geq 1, \qquad u_{k,n}^{(m)} \xrightarrow[n \rightarrow \infty]{} 0.
$$
Then there exists a sequence $(k_n)_{n \geq 1}$ such that $1 \leq k_n \leq n$,
$
k_n \xrightarrow[n \rightarrow \infty]{} \infty,
$
and for any $m \in \N$,
$$
\sum_{i=1}^{k_n} u_{i,n}^{(m)} \xrightarrow[n \rightarrow \infty]{} 0.
$$
\end{lemma}

\begin{proof}[Proof of Lemma \ref{sequence_exists}]
We first prove the statement for one double-indexed sequence $(u_{k,n})_{1 \leq k \leq n}$. We define, for $1 \leq k \leq n$, the partial sums
$
S_{k,n} = \sum_{i=1}^k u_{i,n},
$
and the following sequence, iteratively:
$$
n_1 := 1, \qquad
n_{j+1} := \min \left\{ l \ \mid \ \forall n \geq l, \ 
S_{j+1,n} \leq \frac12 S_{j,n_j} \right\}.
$$
By assumption on $u_{k,n}$, 
the sequence $(n_j)_{j \geq 1}$ is well defined, increasing, and goes to infinity. Moreover, by construction we see that $S_{j,n_j} $ converges to zero.
It is straightforward to check that the sequence
$$
k_n := \max \{  j \in [\![ 1,n ]\!]  \ \mid \ n_j \leq n \}
$$
is such that $1 \leq k_n \leq n$, $k_n \xrightarrow[n \rightarrow \infty]{} \infty$, and
$$
\forall n \in [\![ n_j, n_{j+1} -1 ]\!], \qquad
S_{k_n,n} = S_{j,n} \leq \frac{1}{2} S_{j-1,n_{j-1}},
$$
so that $S_{k_n,n}$ converges to $0$; thus, the Lemma is established for one sequence. We extend this to a countable family of double-indexed sequences $ u_{k,n}^{(m)} $ by defining
$
v_{k,n} := \sum_{m=1}^k u_{k,n}^{(m)},
$
which converges to $0$ for every fixed $k$; by the above argument, there exists a sequence $k_n$ such that
$$
\sum_{j=1}^{k_n} v_{j,n}
\xrightarrow[n \rightarrow \infty]{}
0,
\quad 
\text{so that}
\quad
\forall m \in \N,
\
\sum_{j=1}^{k_n} u_{j,n}^{(m)}
\xrightarrow[n \rightarrow \infty]{} 0.
$$
Indeed, every term being positive, as $k_n \rightarrow \infty$, the latter sum can be bounded by the first one as soon as $k_n \geq m$. This concludes the proof of Lemma \ref{sequence_exists}.
\end{proof}

The following argument uses the multiplicative version of Lemma \ref{sequence_exists}; namely, if a countable family of double-indexed sequences $ p_{k,n}^{(m)}$ is such that $ p_{k,n}^{(m)} \rightarrow 1$ for every fixed $k$ and $m$, then there exists a sequence $(k_n)_{n \geq 1}$, going to infinity, such that for every $m$
$$
\prod_{j=1}^{k_n} p_{j,n}^{(m)} \rightarrow 1.
$$
Note that this existential statement does not give any estimate on the growth rate of $(k_n)$.

\begin{proof}[Proof of Proposition \ref{spherical_gamma2lim}]

We first recall how convergence to $\gamma_2^{-1}$ arises for the complex Ginibre ensemble; part of the argument then relies on comparison with this case, treated in \cite{BourgadeDubach}. The reason why this situation is more tractable is that the distribution of the diagonal overlap yields an exact expression: using a few classical identities of the beta and gamma distributions, we see that
$$
N \Ov_{1,1}^{-1}
\disteq
N \prod_{k=2}^N \left( 1 + \frac{\gamma_1^{(k)}}{\gamma_k} \right)^{-1}
\disteq
N \prod_{k=2}^N \beta_{k,1}
\disteq
N \beta_{2,N-1}
\distconv
\gamma_2.
$$
Now, for any sequence of integers $(k_n)_{n \geq 1}$ such that
\beq\label{sequence_cond}
1 \leq k_n \leq n, \quad
k_n \xrightarrow[n \rightarrow \infty]{} \infty,
\eeq
the same product can be decomposed as
$$
k_N \prod_{k=2}^{k_N} \left( 1 + \frac{\gamma_1^{(k)}}{\gamma_k} \right)^{-1}
\times
\frac{N}{k_N} \prod_{k=k_N+1}^{N} \left( 1 + \frac{\gamma_1^{(k)}}{\gamma_k} \right)^{-1}
\disteq
k_N \beta_{2, k_N-1} \times \frac{N}{k_N} \beta_{k_N+1,N-1}
$$
It is straightforward to check that
$$
k_N \beta_{2, k_N - 1} \distconv \gamma_2, \qquad
\frac{N}{k_N} \beta_{k_N+1,N} \distconv 1.
$$
In other words, the limit distribution $\gamma_2$ essentially depends on the first $k_N$ factors, provided $k_N$ goes to infinity. Similarly in the spherical case, using Theorem \ref{spherical_ov11} and Proposition \ref{Kostlan_conditioned}, we write:
\begin{align*}
N \Ov_{1,1}^{-1}
\disteq
k_N \prod_{k=2}^{k_N} \left( 1+ \frac{1+ \gamma_V(k)}{\gamma_V(k)} X_N^{(k)} \right)^{-1}
& \times
\frac{N}{k_N} \prod_{k=k_N+1}^{N} \left( 1+ \frac{1+ \gamma_V(k)}{\gamma_V(k)} X_N^{(k)} \right)^{-1} \\
=:
F(2,k_N) & \times F (k_N+1,N)
\end{align*}
We will prove that the first factor $F(2,k_N)$ converges to $\gamma_2$ for a suitable sequence $k_N$ that allows comparison with the complex Ginibre case, whereas the second factor $F(k_N+1,N)$ converges to~$1$. By the identity (\ref{spherical_gamma}), the independent variables involved are distributed as follows:
$$
F_{N,k} := 1+ \frac{1+ \gamma_V(k)}{\gamma_{V}(k)} X_N^{(k)}
\disteq
 1+ \frac{ X_N }{\beta_{k,N+1-k}}.
$$
where $X_N$ is defined by (\ref{def_Xm}).

\paragraph{Convergence of $F(2,k_N)$ to $\gamma_2$, for a suitable sequence $(k_N)$.} 
For fixed $k$, each term $F_{N,k}$ converges to its analog in the complex Ginibre case. Indeed,
$$
N X_N \distconv \gamma_1, 
\quad \text{and} \quad
N \beta_{k,N-k+1} \distconv \gamma_k,
$$
so that
$$
F_{N,k}
\disteq
1 + \frac{N X_N}{N \beta_{k,N-k+1}}
\distconv 
1 + \frac{\gamma_1}{\gamma_k}.
$$
The function $x \mapsto x^{-m}$ being smooth and bounded on $(1, \infty)$ for any integer $m$, we have that
$$
\E F_{N,k}^{-m}  \xrightarrow[N \rightarrow \infty]{} \E \left( 1 + \frac{\gamma_1}{\gamma_k} \right)^{-m} 
$$
and so, by the multiplicative version of Lemma \ref{sequence_exists} applied to the appropriate fraction of moments, there exists a sequence $k_n$ verifying (\ref{sequence_cond}), such that for every $m$,
$$
\prod_{k=2}^{k_N}
\frac{\E F_{N,k}^{-m}}{\E \left( 1 + \frac{\gamma_1}{\gamma_k} \right)^{-m}} \rightarrow 1
$$
which implies, by comparison with the product arising in the complex Ginibre case,
$$
\E \left( F(2,k_N)^m \right) =
 k_N^m \E \prod_{k=2}^{k_N} F_{N,k}^{-m} \sim
 k_N^m \E \prod_{k=2}^{k_N} \left( 1 + \frac{\gamma_1}{\gamma_k} \right)^{-m} 
\xrightarrow[]{} (m+1)! = \E \gamma_2^m,
$$
so that we have
$$
F(2,k_N) \distconv \gamma_2.
$$
\paragraph{Convergence of $F(k_N+1,N)$ to the constant $1$.}
Let $k_n$ be the sequence of integers used in the first part of the argument; in particular, it satisfies (\ref{sequence_cond}). We check that this is enough to ensure the convergence of $F(k_N+1,N)$ to $1$. A straightforward computation, similar to the one performed in Proposition \ref{sph_expectation_N}, yields
$$
\E F_{N,k} = \frac{k}{k-1}, \qquad
\E F_{N,k}^2 = \frac{k}{k-2}
$$
so that, thanks to telescopic products, we obtain the following expressions
$$
\E \left( \prod_{k=k_N+1}^N F_{N,k} \right) = \frac{N}{k_N}, 
\qquad
\E \left( \prod_{k=k_N+1}^N F_{N,k}^2 \right) = \frac{N (N-1)}{ k_N (k_N-1)} .
$$
As $k_N$ verifies condition (\ref{sequence_cond}), 
$$
\E \left( F(k_N+1,N)^{-1} \right) = 1,
\qquad
\var \left( F(k_N+1,N)^{-1} \right)
= \frac{N-k_N}{N (k_N-1)}
\rightarrow 0,
$$
which proves that
$ F(k_N+1,N)^{-1} \xrightarrow[N \rightarrow \infty]{L^2} 1 $, and in particular 
$ F(k_N+1,N) \distconv 1 $, concluding the second half of the proof. The claim of the Theorem follows by Slutsky's theorem.\end{proof}

The following proposition relies on the spherical structure of $\Sph(N)$ and has no analog in Section~\ref{tue_section}.

\begin{proposition}\label{spherical_invariance}
The distribution of $\Ov_{1,1}$ conditionally on the event $\{ \la_1 = z \}$, for $z \in \C$, does not depend on $z$.
\end{proposition}

\begin{proof}
Recall that the Jacobian of $p$ at $\la \in \C$ is $\frac{4}{(1+|\lambda|^2)^2}$ and that, for any $\lambda, \mu \in \C$, identity (\ref{stereo_identity}) holds. For any continuous and bounded function $F$ of $N-1$ variables, evaluated in
$$
l_k := \frac{ 4 |\la_1 - \la_k|^2}{ (1+|\la_1 |^2)(1+|\la_k |^2)}
\quad
k=2, \dots, N
$$
we have for any $z \in \C$, by a straightforward change of variables,
\beq\label{general_invariance}
\E_{ \{ \la_1 = z\} } \left( F(l_2, \dots, l_N) \right)
=
\E_{\{ w_1 = p(z) \}} \left( F \left(\| w_1 - w_2 \|^2, \dots, \| w_1 - w_N \|^2 \right) \right),
\eeq
where $(w_1,\dots,w_N)$ is a point process on the sphere with density proportional to (\ref{point_process_1}). As the expectation on the right hand side does not depend on $z$ (by invariance under orthogonal transformations), neither does the one on the left hand side. The claim follows by noting that for any continuous and bounded function $G$, by the tower property of conditional expectation,
$$
\E_{\{ \la_1 = z\}} G(\Ov_{1,1}) = \E_{\{ \la_1 = z\}} F(l_2, \dots, l_N)
$$
where $F(l_2, \dots, l_N) : = \E_{\La} G(\Ov_{1,1}) $ is indeed a function of the variables $l_2, \dots, l_N$.
\end{proof}

Clearly, Propositions \ref{sph_expectation_N}, \ref{spherical_gamma2lim} and \ref{spherical_invariance} provide together a full proof of Theorem \ref{spherical_limit}.

\begin{theorem}\label{spherical_ov12} The quenched expectation of off-diagonal overlaps in the spherical ensemble is given by the formula
\beq
\E_{\La} \left( \Ov_{1,2} \right) = - \frac{1}{N |\la_1 - \la_2|^2}
\prod_{k=3}^N \left( 1+  \frac{(1 + \la_1 \overline{\la_2})(1 + |\la_k|^2)}{N (\la_1 -
\la_k) (\overline{\la_2 - \la_k})}\right)
\eeq
\end{theorem}

\begin{proof} 
Similarly to the diagonal case, we define the partial sums
$$
\Ov_{1,2}^{(d)} : = - \overline{b_2} \sum_{i=2}^d b_i \overline{d_i} .
$$
It follows from the facts presented in Section \ref{overlaps_subsec} that
$$
\Ov_{1,2}^{(2)} = - |b_2|^2 = \frac{-|u_2|^2}{|\la_1 - \la_2|^2}.
$$
One can check, following the proof of Theorem \ref{spherical_ov11}, that $|u_2|^2 \disteq X_N$, so that $$
\E |u_2|^2 =
\frac{1}{N}
\quad \text{and} \quad
\E_{\La} \Ov_{1,2}^{(2)} 
= \frac{- 1}{N |\la_1 - \la_2|^2},
$$
which initiates the recurrence. We now compute the conditional expectation of $b_{n+1} \overline{d_{n+1}}$ by integrating out the vector $u_{n+1}$. We use Proposition \ref{sph_subschur} and (\ref{sph_integral3}) from Lemma
\ref{spherical_integrals} with $a= B_n^*$, $b= D_n^*$ and $S=S_n$
such that $S_n^2 = (1+|\la_n|^2)(I_{n-1} + T_{n-1} T_{n-1}^*)$.
It follows that
$$
\E_{\La, \mathscr{F}_{n-1}} b_{n+1} \overline{d_{n+1}}
=
\frac{1}{N (\la_1 - \la_{n+1}) (\overline{\la_2 - \la_{n+1}})}
B_n S^2 D_n^*
=
\frac{(1+|\la_{n+1}|^2)}{N (\la_1 - \la_{n+1}) (\overline{\la_2 - \la_{n+1}})}
( B_n  D_n^* + B_n T T^* D_n^* )
$$
We notice that, as $T$ is triangular and $B_n$, $D_n$ are subvectors of $L_1$ and $L_2$,
$$
B_n T_n = \la_1 B_n, \quad
D_n T_n = \la_2 D_n,
$$
which gives
$$
- \overline{b}_2 \E_{\La, \mathscr{F}_{n}} b_{n+1} \overline{d_{n+1}}
=
\frac{(1+|\la_{n+1}|^2) (1 + \la_1 \overline{\la_2})}{N (\la_1 - \la_{n+1}) (\overline{\la_2 - \la_{n+1}})}
\Ov_{1,2}^{(n)}.
$$
The factorization follows.
\end{proof}

\begin{proposition}\label{spherical_mixedmom} The conditional expectation of $
\frac{1}{N} \tr G^* G$ with $G$ distributed according to $\Sph(N)$ is given
by the formula:
$$
\E_{\La} \left( \frac{1}{N} \tr G^* G \right)
 =
\prod_{i=1}^N \left( 1+ \frac{ 1+ |\la_i|^2}{N} \right)
-2.
$$
\end{proposition}

\begin{proof} It is clear that $\tr G^* G = \tr T_N^*T_N$, and that for any
$n \leq N$,
$$\tr T_nT_n^* = |\la_n|^2 + \|u_n\|^2 + \tr T_{n-1}T_{n-1}^*,$$
so that defining
$$v_{N,n} = v_{N,n} (\la_1, \dots, \la_n)  : = \E_{N, \La} \tr T_nT_n^*,$$
yields a recursion with $ v_{N,1} = |\la_1|^2 $ and, using Proposition \ref{sph_subschur} and
(\ref{sph_integral4}) from Lemma \ref{spherical_integrals},
\begin{eqnarray}\label{sph_recv}
 v_{N,n+1}
 & = &  v_{N,n} \left(1+ \frac{1+|\la_{n+1}|^2}{N} \right) +
 |\la_{n+1}|^2 +\frac{n}{N} \left(1+ |\la_{n+1}|^2 \right).
\end{eqnarray}
This suggests the introduction of
$
w_{N,n} = v_{N,n} + N + n,
$
for which we see that 
$$
w_{N,1} = N \left( 1+ \frac{1+ |\la_1|^2}{N}\right)
\quad \text{and} \quad
w_{N,n+1} = w_{N,n} \left( 1 + \frac{1+|\la_{n+1}|^2}{N} \right),
$$
so that for every $n \leq N$,
$$
\frac{1}{N} v_{N,n}
 =
\prod_{i=1}^N \left( 1 + \frac{1+|\la_i|^2}{N}\right)
- \left( 1 + \frac{n}{N}  \right)
$$
which is equivalent to the statement, when $n=N$.
\end{proof}

\subsection{Constants and integrals}\label{spherical_estimates}

\begin{lemma}\label{spherical_gammagen}
The normalization constant for generalized gamma variables $\gamma_{V,k}$ with potential $V(x)= M \log(1+x)$ and $1 \leq k \leq M-1$ is
$$
\int_{\R_+} \frac{x^{k-1}}{(1+x)^M} \dd x = \beta(M-k,k),
$$
and $\gamma_{V,k} \disteq \frac{1}{\beta_{M-k,k}}-1.$
Moreover, the associated function $e_V^{(M-2)}$ is given by
$$
e_V^{(M-2)} = (M-1) (1+X)^{M-2} .
$$
\end{lemma}

\begin{proof}
Let us compute, for any suitable function $f$,
\begin{align*}
    \int_{\R_+} \frac{x^{k-1}}{(1+x)^M} f(x) \dd x
    & = \int_0^1 x^{M-k-1} (1-x)^{k-1} f \left( \frac{1}{x}-1 \right) \dd x
    &= \beta (M-k,k) \E f \left( \frac{1}{\beta_{M-k,k}} - 1 \right),
\end{align*}
which implies the first claim. As
$$
\frac{1}{\Gamma_V(k)} = \frac{1}{ \beta(M-k,k) }
= \frac{\Gamma(M)}{\Gamma(M-k) \Gamma(k)}
= (M-k) \binom{M-1}{k-1},
$$
we find that
$$
e_V^{(M-2)} (X) 
= (M-1) (1+X)^{M-1} - X (M-1) (1+X)^{M-2}
= (M-1) (1+X)^{M-2}.
$$
which is the second claim.
\end{proof}

\begin{lemma}\label{spherical_constants}
For any $p>n$,
\begin{align}\label{sph_ct0}
C_{n,p} & := \int_{z \in \C^{n}} \frac{1}{ \left(1 + \sum_{i=1}^{n}
|z_i|^2 \right)^p } \dd m(z_1) \dots \dd m(z_{n}) = \pi^n
\frac{(p-n-1)!}{(p-1)!},
\end{align}
and for $p>n+1$,
\begin{align}\label{sph_ct1}
C_{n,p}^{(1)} & := \int_{z \in \C^{n}} \frac{|z_1|^2}{ \left(1 +
\sum_{i=1}^{n} |z_i|^2 \right)^p } \dd m(z_1) \dots \dd m(z_{n}) =
\frac{1}{p-(n+1)} C_{n,p}.
\end{align}
\end{lemma}

\begin{proof}
We first compute $C_{n,p}$ by induction on $n$. For $n=1, p>1$,
\begin{align*}
C_{1,p} = \int_{z \in \C} \frac{1}{ \left(1 + |z|^2 \right)^p } \dd m(z) =
 \pi \int_{r = 1}^{\infty} \frac{1}{ r^p } \dd r  = \frac{\pi}{p-1}
\end{align*}
and one can note that for any $\al>0$,
\begin{align*}
\int_{z \in \C} \frac{1}{ \left(1 + \al^{-1} |z|^2 \right)^p } \dd m(z) =
\frac{\pi \al}{p-1}.
\end{align*}
For general $n$, using the above equalities with $\al_n = 1+
\sum_{i=1}^{n-1} |z_i|^2$,
\begin{align*}
C_{n,p}
& = \int_{z \in \C^n} \frac{1}{ \left(1 + \sum_{i=1}^{n-1} |z_i|^2
\right)^p }
\times
\frac{1}{ \left(1 + \al_n^{-1} |z_n|^2 \right)^p }
\dd m(z_1) \dots \dd m(z_n) \\
&= \frac{\pi}{p-1} \int_{z \in \C^{n-1}} \frac{1}{ \left(1 +
\sum_{i=1}^{n-1} |z_i|^2 \right)^{p-1} }
\dd m(z_1) \dots \dd m(z_{n-1})
= \frac{\pi}{p-1} C_{n-1,p-1}
\end{align*}
Equation (\ref{sph_ct0}) follows. A similar induction can be performed on
$C_{n,p}^{(1)}$. The only difference is that the last step involves the
following identity: for any $p>2$,
\begin{align*}
C_{1,p}^{(1)} & = \int_{z \in \C} \frac{|z|^2}{ (1 + |z|^2)^p } \dd m(z)
= \pi \int_{r = 1}^{\infty} \frac{r-1}{ r^p } \dd r
= \pi \left( \frac{1}{1-p} - \frac{1}{2-p} \right)
= \frac{\pi}{(p-1)(p-2)},
\end{align*}
which, in general, yields the extra factor $\frac{1}{p-(n+1)}$ in
(\ref{sph_ct1}).
\end{proof}

Note that when we begin the recursion from \cite{ForresterKrishnapur} with
$n=N-1,p=2N$, the extra factor is $\frac{1}{N}$ at every step.

\begin{lemma}\label{spherical_integrals}
For any $p>n$, $a,b \in \C^n$ and any Hermitian positive-definite matrix $S$,
\begin{align}
\int_{\C^n} \frac{ 1 }{(1+ u^* S^{-2} u)^{p}} \dd u
& = C_{n,p} |\det S|^2, \label{sph_integral1} \\
\int_{\C^n} \frac{ a^*u }{(1+ u^* S^{-2} u)^{p}} \dd u
& = 0, \label{sph_integral2} \\
\int_{\C^n} \frac{ (a^* u) (u^* b) }{(1+ u^* S^{-2} u)^{p}} \dd u
&= C_{n,p} |\det S|^2 \frac{ a^* S^2 b }{p-(n+1)}, \label{sph_integral3} \\
\int_{\C^n} \frac{ \|u\|^2 }{(1+ u^* S^{-2} u)^{p}} \dd u
&= C_{n,p} |\det S|^2 \frac{ \tr S^2 }{p-(n+1)}, \label{sph_integral4}
\end{align}
where the constant $C_{n,p}$ is explicitly computed in Lemma
\ref{spherical_constants}.
\end{lemma}

\begin{proof} Integral (\ref{sph_integral1}) was computed in
\cite{ForresterKrishnapur}. (\ref{sph_integral2}) is zero by symmetry. For
(\ref{sph_integral3}), the change of variables $u=Sv$ yields
$$
 |\det S|^2 \int \dd v \frac{ (a^* S v) (v^* S b)}{ (1+ v^* v)^{p}}.
$$
We notice that
$$
(a^* S v) (v^* S b) = v^* (S b a^* S) v
= v^* A v
$$
where $A = S b a^* S$ is a matrix of rank $1$. If we express $v = \sum v_i e_i$ in a unitary basis such that the vectors $(e_2, \dots, e_n)$ form a basis of $\ker(A)$ and denote $A e_1 = \la_1(A) e_1 + \sum_{i \geq 2} \al_i e_i$,
$$
v^* A v 
= \la_1(A) v_1^2 + \sum_{i \geq 2} \al_i v_1 v_i
$$
Therefore, after a unitary change of basis the integral becomes, using Lemma \ref{spherical_constants} and the fact that cross-terms $v_1 v_i$ vanish by symmetry,
$$
\int \dd v \frac{ \la_1(A) v_1^2}{ (1+ v_1^2 + \dots + v_n^2)^{p}}
=
\frac{ \la_1(A) }{p-(n+1)} C_{n,p}.
$$
The value of $\la_1(A)$ can be obtained by writing
$$
\la_1(A) = \tr S b a^* S = a^* S^2 b,
$$
from which the claim (\ref{sph_integral3}) follows. The same technique
applied to (\ref{sph_integral4}) yields
$$
|\det S|^2 \int \dd v \frac{ \| Sv \|^2}{ (1+ v^* v)^{p}}.
$$
and a unitary change of variable to a basis that diagonalizes $S$,
together with Lemma \ref{spherical_constants}, gives
$$
\int \dd v \frac{ \la_1(S^2) v_1^2 + \dots + \la_n(S^2) v_n^2}{ (1+ v_1^2
+ \dots + v_n^2)^{p}}
= (\la_1(S^2) + \cdots + \la_n(S^2)) \frac{1}{p-(n+1)} C_{n,p},
$$
concluding the proof of the last claim.
\end{proof}

\begin{lemma}\label{spherical_distributions}
For any $p>n$, $a \in \C^n$ and any Hermitian positive-definite matrix $S$, if $u \in \C^n$ is distributed with density
$$
\frac{1}{C_{n,p} |\det S|^2}
\frac{1}{ ( 1 + u^* S^{-2} u )^p}
$$
with respect to the Lebesgue measure on $\C^n$, then the following identity in distribution holds:
$$
| a^*u |^2 \disteq \|S a\|^2 X_{p-n-1}.
$$
\end{lemma}

\begin{proof}
By a direct change of variable, it is clear that $u \disteq S \bV_p^{n}$. We note that $|a^*Sv|^2 = v^* A v$ where $A=Saa^*S$ is a Hermitian matrix of rank one. A unitary change of variable brings it to the form $\la_1(A) v_1^2$ with $\la_1(A) = \tr A = a^*S^2a = \|Sa \|^2$. Successive integration of the other coordinates $v_2, \dots, v_n$ yields the result. \end{proof}

\newpage
\section{Truncated unitary ensemble}\label{tue_section}

This section contains the proof of all claims concerning the truncated unitary ensembles $\TUE(N,M)$ when $N \leq M$. Almost every step in this study is analogous to what was done in the spherical case; we therefore refer constantly to the corresponding parts of Section \ref{spherical_section}. 

\subsection{Schur form and eigenvalues}

As in Section \ref{spherical_section}, we first present a few general results in order to illustrate the method, as well as a few tools and definitions that are specific to the truncated unitary case. We first recall that the Schur transfom $T$ is distributed with density proportional to 

\beq\label{tue_Schur_density}
\prod_{i<j} | \la_i - \la_j |^2
\det(I_N - T T^*)^{M-N} \mathds{1}_{T T^* < 1}
\eeq
with respect to the Lebesgue measure on all complex matrix elements, diagonal ($\dd \La = \dd \la_1 \cdots \dd \la_N$) and upper-triangular ($\dd u_2 \cdots \dd u_n$).  \medskip

Provided $TT^* <1$ (which implies the same condition on every submatrix $T_n$), we introduce the Hermitian, definite-positive matrices
\beq\label{tue_def_HnSn}
H_n := I_n-T_n T_n^*,
\qquad
S_{n-1} := (1-|\la_n|^2)^{1/2} H_{n-1}^{1/2}.
\eeq
Note that the only differences with the matrices $H_n, S_{n-1}$ used in the spherical case are the minus sign and the condition on the eigenvalues of $T T^*$.

\begin{lemma}\label{tue_FK_trick} The determinant of $H_n = I_n-T_n T_n^*$ can be reccursively decomposed as
\begin{align}
\det(H_n)
= &
(1-|\la_n|^2)
\det(H_{n-1})
\left( 1 - \frac{1}{1-|\la_n|^2} u_n^* H_{n-1}^{-1} u_n \right).
\end{align}
\end{lemma}

The proof is analogous to the proof of Lemma \ref{FK_trick}.

For any $p \geq 0$, we denote by $\bW_p^{(n)}$ a random vector with density
\beq\label{def_Wpn}
\frac{1}{C_{n,p}}
{(1-v^* v)^p}
\mathds{1}_{v^* v < 1}
\eeq
with respect to the Lebesgue measure on $\C^n$; the value of $C_{n,p}$ is given by (\ref{tue_ct0}). For any $m \geq 2$, we denote by $Y_m$ a real random variable with density
\beq\label{def_Ym}
(m-1) (1-y)^{m-2} \mathds{1}_{(0,1)}
\eeq
with respect to the Lebesgue measure, i.e. it follows a $\beta_{1,m-1}$ distribution; in particular $\E Y_m = \frac{1}{m}$. If $w_i$ is a coordinate of $\bW_p^{(n)}$, it follows from Lemma \ref{tue_distributions} that
$$
|w_i|^2 \disteq Y_{p+n+1}.
$$
Note that the i.i.d. variables that appear in Theorem \ref{tue_ov11} follow the above distribution with $m=M$.

\begin{lemma}\label{tue_key_prop}
Identity holds between the following expressions, for $p \geq n$ and $f, g$ integrable functions of the matrix elements: 
\begin{align*}
& \int {f(T_{n-1}, \la_n) g (u_{n})} {\det ( H_n )^{p}} \mathds{1}_{T_n T_n^* < 1}
 \dd T_n \\
 =& 
C_{n-1,p}
\int {f(T_{n-1}, \la_n) \E \left( g(S_{n-1} \bW_p^{(n-1)}) \right) }
{(1-|\la_n|^2)^{p+n-1} \det ( H_{n-1} )^{p+1}}
\mathds{1}_{T_{n-1} T_{n-1}^* < 1}
\dd T_{n-1} \dd \lambda_{n},
\end{align*}
where
$ H_{n}, S_{n-1}, \bW_p^{(n)} $
are defined in (\ref{def_HnSn}) and (\ref{def_Vpn}).
\end{lemma}

We deduce from the above Lemma the distribution of every top-left submatrix of the Schur form, analogously to Proposition \ref{sph_subschur}.

\begin{proposition}\label{tue_subschur}
Conditionally on $\La$ and for $2 \leq n \leq N$, the submatrix $T_n$ of the Schur transform is distributed with density proportional to
\beq
\det(I_n - T_n T_n^*)^{M-n} \mathds{1}_{T_nT_n^* \leq 1}.
\eeq
with respect to the Lebesgue measure on upper-triangular matrix elements ($\dd u_2 \cdots \dd u_n$).
\end{proposition}

We also derive the joint eigenvalue density of the truncated unitary ensemble from the density of its Schur form, as was done in \cite{ForresterKrishnapur}. The result itself was first proven in \cite{Sommers}.

\begin{theorem}[\.Zyczkowski \& Sommers]\label{tue_joint_density}
The joint density of eigenvalues for the truncated unitary ensemble when $M \geq N$ is proportional to
\beq\label{tue_eigdensity}
\frac{1}{Z_{M,N}}
\prod_{1 \leq i<j \leq N} |\la_i - z_j|^2
\prod_{i=1}^N (1-|\la_i|^2)^{M-1}
\mathds{1}_{\D}(\la_i)
\eeq
with respect to the Lebesgue measure on $\C^N$.
\end{theorem}

The proof is analogous to the one of Theorem \ref{sph_joint_density}. \medskip

Theorem \ref{tue_joint_density} can be rephrased by saying that the eigenvalues of $\TUE(N,M)$ are distributed according to (\ref{density_general_V}) with potential $V(t)= -(M-1) \ln (1-t) \mathds{1}_{(0,1)}$. A straightforward computation shows that
\beq
\ga_V(\al) \disteq \beta_{\al, M}.
\eeq

Thus, Kostlan's theorem in that case asserts that the set of squared radii is distributed as a set of independent $\beta$ variables. Namely,
$$
\{ |\lambda_1|^2, \dots, |\lambda_k|^2 \}
\stackrel{d}{=}
\{ \beta_{1,M}, \dots, \beta_{k,M} \}.
$$


\subsection{Distribution and conditional expectation of overlaps}

\begin{theorem}\label{tue_ov11} 
Conditionally on $\{ \La = (\la_1, \dots, \la_N) \}$, diagonal overlaps in the truncated unitary ensemble $\TUE(N,M)$ are distributed as
\beq
\Ov_{1,1} \disteq \prod_{k=2}^N \left( 1+  \frac{(1 -
|\la_1|^2)(1 - |\la_k|^2)}{|\la_1 - \la_k|^2} Y_{M}^{(k)} \right)
\eeq
where the $ Y_{M}^{(k)}$ are i.i.d. distributed according to (\ref{def_Ym}) with $m=M$. In particular, the quenched expectation is given by the formula
\beq
\E_{\La} \left( \Ov_{1,1} \right) = \prod_{k=2}^N \left( 1+  \frac{(1 -
|\la_1|^2)(1 - |\la_k|^2)}{M |\la_1 - \la_k|^2}\right)
\eeq
\end{theorem}

\begin{proof} It is similar to the one of Theorem \ref{spherical_ov11}; we sketch it again to see where the differences lie. We first write
$$
\Ov_{1,1}^{(d+1)} = \Ov_{1,1}^{(d)} + |b_{d+1}|^2 = \Ov_{1,1}^{(d)} \left(1 + \frac{1}{|\la_1 - \la_{d+1}|^2} \frac{ | B_d u_{d+1}  |^2 }{\| B_d \|^2} \right) 
$$
In order to characterize the distribution of this factor, we use Proposition \ref{tue_subschur}, then Lemma \ref{tue_key_prop} and Lemma \ref{tue_distributions} with $a=b= \overline{B_d}$ and $S=S_{d+1}$ such that $S_{d+1}^2 = (1-|\la_{d+1}|^2)(I_{d} - T_{d} T_{d}^*)$. This yields
\beq
| B_d u_{d+1} |^2 \disteq (1-|\la_{d+1}|^2) \| (I_d - T_d T_{d}^*) \overline{B_d} \|^2 Y_{N}
\eeq
where $Y_{N}$ is distributed according to (\ref{def_Ym}) with $m=M$, and independent of $\mathscr{F}_d$; we denote this variable by $Y_{N}^{(d+1)}$ to avoid confusion. The last steps of the proof follow accordingly.
\end{proof}

\begin{proposition}\label{tue_expectation_N}
Conditionally on $\{ \la_1=0 \}$, the expectation of the diagonal overlap $\Ov_{1,1}$ in the truncated unitary ensemble $\TUE(N,M)$ is
$$
\E_{\{ \la_1 = 0 \}} \Ov_{1,1} = N.
$$
\end{proposition}

Note that the same statement, which is an exact identity for any $N$, holds in the complex Ginibre ensemble and spherical ensemble respectively.

\begin{proof}
We know from Proposition \ref{Kostlan_conditioned} that the squared radii, conditionally on the event $\{ \la_1=0 \}$, are distributed like independent variables with distributions $\gamma_{V,k}$ with $V (x)= -(M-1) \log(1-x) \mathds{1}_{(0,1)}$ and $2 \leq k \leq N$. We already noticed that $\gamma_{V,k} \disteq \beta_{k,M}$. A straightforward computation follows:
$$
\E_{ \{ \la_1 = 0 \} } \Ov_{1,1}
=
\E \prod_{k=2}^N \left( 1 + \frac{1- |\la_k|^2}{M |\la_k|^2} \right)
= \prod_{k=2}^N \E \left( 1 - \frac{1}{M} + \frac{1}{M \beta_{k,M}} \right).
$$
For any $k \geq 2$,
$$
\E \left( \frac{1}{ \beta_{k,M}} \right)
= \frac{\beta (k-1, M) }{ \beta(k,M) }
= \frac{M +k - 1}{k-1},
$$
so that the expectation is given by the telescopic product
$$
\E_{\{ \la_1 = 0 \}} \Ov_{1,1}
= \prod_{k=2}^N \frac{k}{k-1}
=N
$$
as was claimed.
\end{proof}

\begin{proposition}\label{tue_gamma2lim}
Conditionally on $\{ \la_1=0 \}$, the following convergence in distribution takes place:
$$
\frac{1}{N} \Ov_{1,1} \distconv \frac{1}{\gamma_2}.
$$
\end{proposition}

Note that $N \rightarrow \infty$ implies $M \rightarrow \infty$, as we study the truncated unitary ensemble in the regime where $N \leq M$. The rate at which $N,M$ go to infinity does not have any impact on the following proof (although it is expected to play a role when conditioning on a generic $z$ in the bulk).

\begin{proof} The technique is similar to the proof of Proposition \ref{spherical_gamma2lim}. We decompose the distribution obtained by Theorem \ref{tue_ov11} in two factors
\begin{align*}
N \Ov_{1,1}^{-1}
\disteq
k_N \prod_{k=2}^{k_N} \left( 1+ \frac{1 - \gamma_V(k)}{\gamma_V(k)} Y_M^{(k)} \right)^{-1}
& \times
\frac{N}{k_N} \prod_{k=k_N+1}^{N} \left( 1+ \frac{1 - \gamma_V(k)}{\gamma_V(k)} Y_M^{(k)} \right)^{-1} \\
=:
G(2,k_N) & \times G(k_N+1,N).
\end{align*}
As $\gamma_V (k) \disteq \beta_{k,M}$, we have
$$
G_{M,k} : = 1+ \frac{1 - \gamma_V(k)}{\gamma_V(k)} Y_M^{(k)}
\disteq
1+ \left(\frac{1}{\beta_{k,M}} - 1 \right) Y_M,
$$
where $Y_M$ is defined by (\ref{def_Ym}). The proof then proceeds in two separate parts.

\paragraph{Convergence of $G(2,k_N)$ to $\gamma_2$ for a suitable sequence $k_N$.}
It is straightforward to check that for every $k$, the term $G_{M,k}$ converges to the factor playing an analogous role in the complex Ginibre case. Indeed,
$$
M Y_M \distconv \gamma_1, 
\quad \text{and} \quad
M \beta_{k,M} \distconv \gamma_k,
$$
so that
$$
G_{M,k}
\disteq
1 + \left( \frac{1}{M \beta_{k,M}} - \frac{1}{M} \right) M Y_M
\distconv 
1 + \frac{\gamma_1}{\gamma_k}.
$$
The argument then proceeds exactly as in Proposition \ref{spherical_gamma2lim}: by Lemma \ref{sequence_exists}, there exists a sequence $k_N$ that verifies (\ref{sequence_cond}) and such that we can derive the convergence
$$
G(2,k_N) \distconv \gamma_2
$$
by comparison with the complex Ginibre case.

\paragraph{Convergence of $G(k_N+1,N)$ to $1$.}
It follows from the computation performed in the proof of Proposition \ref{tue_expectation_N} that
$$
\E G_{M,k} = \frac{k}{k-1},
$$
which is the same as the expectation of $F_{N,k}$ (and does not depend on $M$ nor $N$). We compute the second moment, using the values
$$
\E Y_M^2 = \frac{2}{M(M+1)},
\quad \text{and} \quad
\E \left( \frac{1}{\beta_{k,M}} -1 \right)^2
= \frac{M (M+1)}{(k-1)(k-2)}
$$
and find, as for $F_{N,k}$,
$$
\E G_{M,k}^2 = \frac{k}{k-2}
$$
so that we obtain the exact same expressions as in the spherical case. The end of the argument (and of the whole proof) is strictly similar to what has been written in the proof of Proposition \ref{spherical_gamma2lim}.
\end{proof}

The analog of the spherical structure of $\Sph(N)$ for $\TUE(N,M)$ is the stereographic projection on the pseudosphere (see \cite{ForresterKrishnapur}). However, the symmetries of the pseudosphere do not allow to establish an exact equivalent to Proposition \ref{spherical_invariance}.

\begin{theorem}\label{tue_ov12} The quenched expectation of off-diagonal overlaps in $\TUE(N,M)$ with $N \leq M$ is given by the formula
\beq
\E_{\La} \left( \Ov_{1,2} \right) = - \frac{1}{M |\la_1 - \la_2|^2}
\prod_{k=3}^N \left( 1+  \frac{(1 - \la_1 \overline{\la}_2)(1 - |\la_k|^2)}{M (\la_1 -
\la_k) (\overline{\la_2 - \la_k})}\right)
\eeq
\end{theorem}

\begin{proof} As for the proof of theorem \ref{spherical_ov12}, we consider the partial sums
$
\Ov_{1,2}^{(d)}
$
and proceed by induction.
It follows from the proof of Theorem \ref{spherical_ov11}, that $|u_2|^2 \disteq Y_M$, so that $$
\E |u_2|^2 =
\frac{1}{M}
\quad \text{and} \quad
\E_{\La} \Ov_{1,2}^{(2)} 
= \frac{- 1}{M |\la_1 - \la_2|^2}.
$$
We then compute the conditional expectation of $b_{n+1} \overline{d_{n+1}}$ by integrating out the vector $u_{n+1}$, using Proposition \ref{tue_subschur} and (\ref{tue_integral3}) from Lemma
\ref{tue_integrals} with $a= B_n^*$, $b= D_n^*$ and $S=S_n$.
It follows that
$$
\E_{\La, \mathscr{F}_{n-1}} b_{n+1} \overline{d_{n+1}}
=
\frac{(1 - |\la_{n+1}|^2)}{M (\la_1 - \la_{n+1}) (\overline{\la_2 - \la_{n+1}})}
( B_n  D_n^* - B_n T_n T_n^* D_n^* )
$$
As noted in the proof of Theorem \ref{spherical_ov12}, we have
$$
B_n T_n = \la_1 B_n, \quad
D_n T_n = \la_2 D_n,
$$
and conclude that
$$
- \overline{b}_2
\E_{\La, \mathscr{F}_{n}} b_{n+1} \overline{d_{n+1}}
=
\frac{(1-|\la_{n+1}|^2) (1 - \la_1 \overline{\la_2})}{M (\la_1 - \la_{n+1}) (\overline{\la_2 - \la_{n+1}})}
\Ov_{1,2}^{(n)}
$$
and the factorization follows.
\end{proof}

\begin{proposition}\label{tue_mixedmom} The quenched expectation of $\tr G^* G$ with $G$ distributed according to $\TUE(N,M)$ is given by the formula:
$$
\E_{\La} \left( \frac{1}{N} \tr G^* G \right)
 =
\prod_{i=1}^N \left( 1 + \frac{1-|\la_i|^2}{M} \right)
- \left( 1+ \frac{N}{M}\right).
$$
\end{proposition}

\begin{proof} As in the proof of Proposition \ref{spherical_mixedmom}, we define $v_{N,n} : = \E_{N, \La} \tr T_nT_n^*$ and note that for any $n \leq N$,
$$\tr T_nT_n^* = |\la_n|^2 + \|u_n\|^2 + \tr T_{n-1}T_{n-1}^*.$$
Using (\ref{tue_integral4}) from Lemma \ref{tue_integrals} yields a induction with $ v_{N,1} = |\la_1|^2 $ and
\begin{eqnarray}\label{tue_recv}
 v_{N,n+1}
 & = &  v_{N,n} \left(1+ \frac{1-|\la_{n+1}|^2 }{M} \right)  + |\la_{n+1}|^2
 +\frac{n}{M} \left(1 - |\la_{n+1}|^2 \right).
\end{eqnarray}
This is an analogous recursion formula to the one obtained in Proposition \ref{spherical_mixedmom} and it can be solved the same way, replacing $|\la_i|^2$ by $-|\la_i|^2$ and $N$ by $M$ in the denominators; this leads to the expression
$$
\frac{1}{N} v_{N,n}
 =
\prod_{i=1}^n \left( 1 + \frac{1-|\la_i|^2}{M} \right)
- \left( 1 + \frac{n}{M}  \right)
$$
which is equivalent to the statement, when $n=N$.
\end{proof}

\subsection{Constants and integrals}\label{tue_estimates}

\begin{lemma}\label{tue_constants}
For any $p \geq 0$, with $\bB_n : = \{ (\la_1, \dots, \la_n) \in \C^n \ | \ \sum
|\la_i|^2 \leq 1 \}$,
\begin{align}\label{tue_ct0}
D_{n,p} & := \int_{z \in \bB_n} { \left(1 - \sum_{i=1}^{n} |\la_i|^2 \right)^p
} \dd m(\la_1) \dots \dd m(z_{n}) = \pi^n \frac{p!}{(p+n)!},
\end{align}
and
\begin{align}\label{tue_ct1}
D_{n,p}^{(1)} & := \int_{z \in \bB_n} {|\la_1|^2} { \left(1 - \sum_{i=1}^{n}
|\la_i|^2 \right)^p } \dd m(\la_1) \dots \dd m(z_{n}) = \frac{1}{p+n+1}
D_{n,p}.
\end{align}
\end{lemma}

\begin{proof}
We first compute $D_{n,p}$ by induction on $n$. For $n=1, p \geq 0$,
\begin{align*}
D_{1,p} = \int_{z \in \D} { \left(1 - |z|^2 \right)^p } \dd m(z) =  \pi
\int_{r = 0}^{1} { r^p } \dd r  = \frac{\pi}{p+1}
\end{align*}
note that for any $\al>0$,
\begin{align*}
\int_{|z|^2 < \alpha} { \left(1 - \al^{-1} |z|^2 \right)^p } \dd m(z) =
\frac{\pi \al}{p+1}.
\end{align*}
For general $n$, using the above equalities with $\al_n = 1 -
\sum_{i=1}^{n-1} |\la_i|^2$,
\begin{align*}
D_{n,p}
& = \int_{z \in \bB_n} { \left(1 - \sum_{i=1}^{n-1} |\la_i|^2 \right)^p }
\times
{ \left(1 - \al_n^{-1} |\la_n|^2 \right)^p }
\dd m(\la_1) \dots \dd m(\la_n) \\
&= \frac{\pi}{p+1} \int_{z \in \bB_{n-1}} { \left(1 - \sum_{i=1}^{n-1}
|\la_i|^2 \right)^{p+1} }
\dd m(\la_1) \dots \dd m(\la_{n-1})
= \frac{\pi}{p+1} D_{n-1,p+1}.
\end{align*}
Equation (\ref{tue_ct0}) follows. A similar induction can be performed on
$D_{n,p}^{(1)}$. The only difference is that the last step involves the
following identity: for any $p \geq 0$,
\begin{align*}
D_{1,p}^{(1)} & = \int_{z \in \D} {|z|^2}{ (1 - |z|^2)^p } \dd m(z)
= \pi \int_{r = 0}^{1} {(r-1)}{ r^p } \dd r
= \pi \left( \frac{1}{p+1} - \frac{1}{p+2} \right)
= \frac{\pi}{(p+1)(p+2)},
\end{align*}
which in general yields the extra factor $\frac{1}{p+n+1}$ in
(\ref{tue_ct1}).
\end{proof}

Note that when we begin the recursion from \cite{ForresterKrishnapur} with
$n=N-1,p=M-N$, the extra factor is $\frac{1}{M}$ at every step.

\begin{lemma}\label{tue_integrals}
For any $p>n$, $a,b \in \C^N$ and any Hermitian positive-definite matrix $S$,
\begin{align}
\int_{S \bB_n} {(1 - u^* S^{-2} u)^{p}} \dd u
& = D_{n,p} |\det S|^2, \label{tue_integral1} \\
\int_{S \bB_n} { (a^*u) }{(1- u^* S^{-2} u)^{p}} \dd u
& = 0, \label{tue_integral2} \\
\int_{S \bB_n} {(a^* u) (u^* b) }{(1 - u^* S^{-2} u)^{p}} \dd u
&= D_{n,p} |\det S|^2 \frac{ a^* S^2 b }{n+p+1}, \label{tue_integral3} \\
\int_{S \bB_n} { \|u\|^2 }{(1 - u^* S^{-2} u)^{p}} \dd u
&= D_{n,p} |\det S|^2 \frac{ \tr S^2 }{n+p+1}, \label{tue_integral4}
\end{align}
where the constant $D_{n,p}$ is explicitly computed in Lemma \ref{tue_constants}.
\end{lemma}


\begin{lemma}\label{tue_distributions}
For any $p>n$, $a \in \C^n$ and any Hermitian positive-definite matrix $S$, if $u \in \C^n$ is distributed with density
$$
\frac{1}{C_{n,p} |\det S|^2} ( 1 - u^* S^{-2} u )^p
$$
with respect to the Lebesgue measure on $\C^n$, then the following identity in distribution holds:
$$
| a^*u |^2 \disteq \|S a\|^2 Y_{p+n+1}.
$$
\end{lemma}

The proofs of Lemmata \ref{tue_integrals}  and \ref{tue_distributions} are exactly analogous to the
proofs of their spherical counterpart, Lemmata \ref{spherical_integrals} and \ref{spherical_distributions}.

\end{document}